\documentclass[12pt]{amsart}
\usepackage{amssymb,amsfonts,amsthm,latexsym,bm}
\usepackage{hyperref}
\usepackage[centertags]{amsmath}
\usepackage{mathrsfs}
\usepackage{t1enc}
\usepackage{graphicx}
\usepackage{enumerate}

\advance\headheight by 3pt

\newtheorem{thm}{Theorem}[section]
\newtheorem{cor}[thm]{Corollary}

\newtheorem*{theorem*}{Theorem}
\newtheorem*{subspacetheorem*}{Subspace Theorem}

\numberwithin{equation}{section}

\def\pitem{\advance\leftskip3mm\advance\linewidth-3mm}
\def\mitem{\advance\leftskip-3mm\advance\linewidth3mm}

\begingroup\makeatletter\ifx\SetFigFont\undefined
\gdef\SetFigFont#1#2#3#4#5{
  \reset@font\fontsize{#1}{#2pt}
  \fontfamily{#3}\fontseries{#4}\fontshape{#5}
  \selectfont}
\fi\endgroup

\renewcommand{\subjclass}[1]{\thanks{\emph{2010 Mathematics Subject Classification:}~#1}}
\renewcommand{\keywords}[1]{\thanks{\emph{Keywords and Phrases:}~#1}}
\renewcommand{\date}{\thanks{\today}}

\newcommand{\av}{{\bf a}}

\newcommand{\ev}{{\bf e}}

\newcommand{\uv}{{\bf u}}
\newcommand{\vv}{{\bf v}}

\newcommand{\xv}{{\bf x}}
\newcommand{\yv}{{\bf y}}
\newcommand{\zv}{{\bf z}}
\newcommand{\Av}{{\bf A}}
\newcommand{\Bv}{{\bf B}}

\newcommand{\Xv}{{\bf X}}

\newcommand{\nullv}{{\bf 0}}

\newcommand{\CC}{\mathcal{C}}

\newcommand{\EE}{\mathcal{E}}

\newcommand{\II}{\mathcal{I}}

\renewcommand{\SS}{\mathcal{S}}

\newcommand{\Cc}{\mathbb{C}}
\newcommand{\Rr}{\mathbb{R}}
\newcommand{\Qq}{\mathbb{Q}}

\newcommand{\Zz}{\mathbb{Z}}

\def\house#1{\setbox1=\hbox{$\,#1\,$}%
\dimen1=\ht1 \advance\dimen1 by 2pt \dimen2=\dp1 \advance\dimen2
by 2pt
\setbox1=\hbox{\vrule height\dimen1 depth\dimen2\box1\vrule}%
\setbox1=\vbox{\hrule\box1}%
\advance\dimen1 by .4pt \ht1=\dimen1 \advance\dimen2 by .4pt
\dp1=\dimen2 \box1\relax}

\newcommand{\kdots}{,\ldots ,}

\newcommand{\medfrac}[2]{\mbox{\large{$\textstyle{\frac{#1}{#2}}$}}}
\newcommand{\medbinom}[2]{\mbox{\large{$\textstyle{\binom{#1}{#2}}$}}}

\title[Mahler's work on the geometry of numbers]
{Mahler's work on the geometry of numbers}
\date

\author[J.-H. Evertse]{Jan-Hendrik Evertse}
\address{J.-H. Evertse \newline
         \indent Universiteit Leiden, Mathematisch Instituut, \newline
         \indent Postbus 9512, 2300 RA Leiden, The Netherlands}
\email{evertse\char'100math.leidenuniv.nl}

\begin{document}

\maketitle

Mahler has written many papers on the geometry of numbers. 
Arguably, his most influential achievements in this area are his compactness
theorem for lattices, his work on star bodies
and their critical lattices, and his estimates
for the successive minima of reciprocal convex bodies and compound
convex bodies. We give a, by far not complete, overview of Mahler's work
on these topics and their impact.

\section{Compactness theorem, star bodies and 
their critical lattices}\label{section1}

Many problems in the geometry of numbers are about whether a particular
$n$-dimensional body contains a non-zero point from a given lattice,
and quite often one can show that this is true as long as the determinant
of the lattice is below a critical value depending on the given body.
Mahler intensively studied such problems for so-called \emph{star bodies}.
Before mentioning some of his results, 
we start with recalling some definitions.
We follow \cite{Mahler1946d}.

Let $n\geq 2$ be an integer that we fix henceforth.
A \emph {distance function} on $\Rr^n$
is a function $F:\, \Rr^n\to\Rr$ such that:
\begin{itemize}
\item[(i)]
$F(\xv )\geq 0$ for all $\xv\in\Rr^n$ and $F(\xv)>0$ for at least one $\xv$;
\item[(ii)] $F(t\xv )=|t|\cdot F(\xv )$ for $\xv\in\Rr^n$ and $t\in\Rr$;
\item[(iii)] $F$ is continuous.
\end{itemize}
A \emph{(symmetric) star body} in $\Rr^n$ is a set of the shape 
\[
\SS =\{ \xv\in\Rr^n:\, F(\xv )\leq 1\}, 
\]
where $F$ is a distance function. 
We call $\SS$ the star body with distance function $F$.
The boundary of $\SS$ is $\{ \xv\in\Rr^n:\, F(\xv )=1\}$,
and the interior of $\SS$ is $\{ \xv\in\Rr^n:\, F(\xv )<1\}$. 
The set $\SS$ is bounded, if and only if $F(\xv)>0$ whenever $\xv\not= 0$.
The star bodies contain
as a subclass the \emph{symmetric convex bodies}, which correspond
to the distance functions $F$ satisfying in addition to (i),(ii),(iii)
the triangle inequality $F(\xv +\yv )\leq F(\xv )+F(\yv)$ for $\xv,\yv\in\Rr$.

Let $\Lambda =\{ \sum_{i=1}^n z_i\av_i:\, z_1\kdots z_n\in\Zz\}$ 
be a lattice in $\Rr^n$
with basis $\{\av_1\kdots\av_n\}$. We define its determinant by
$d(\Lambda ):=|\det (\av_1\kdots\av_n )|$.
Let $\SS$ be a star body.
We call $\Lambda$ $\SS$-\emph{admissible} if $\nullv$ is the only point of $\Lambda$ in the interior of $\SS$.
The star body $\SS$ is called of \emph{finite type} if it has admissible lattices,
and of infinite type otherwise. Bounded star bodies are necessarily of
finite type, but conversely, star bodies of finite type do not have to be
bounded. For instance, let 
$\SS :=\{ \xv=(x_1\kdots x_n)\in\Rr^n:\, |x_1\cdots x_n|\leq 1\}$.
Take a totally real number field $K$ of degree $n$, denote by $O_K$
its ring of integers, and let $\alpha\mapsto\alpha^{(i)}$ $(i=1\kdots n)$ be the
embeddings of $K$ in $\Rr$. Then $\{ (\alpha^{(1)}\kdots\alpha^{(n)}):\, \alpha\in O_K\}$ is an $\SS$-admissible lattice. 

Assume henceforth that $\SS$ is 
a star body of finite type.
Then we can define its determinant,
\[
\Delta (\SS ):=\inf\{ d(\Lambda ):\,\Lambda\ \mbox{admissible lattice for }\SS\}.
\]
Thus, if $\Lambda$ is any lattice in $\Rr^n$ with $d(\Lambda )<\Delta (\SS )$,
then $\SS$ contains a non-zero point from $\Lambda$.
The quantity $\Delta (\SS )$ cannot be too small. From the Minkowski-Hlawka
theorem (proved by Hlawka 
\cite{Hlawka1944} and earlier stated without proof by Minkowski)
it follows that $\Delta (\SS )>(2\zeta (n))^{-1}V(\SS)$, where 
$\zeta (n)=\sum_{k=1}^{\infty} k^{-n}$ and $V(\SS )$ is the volume
($n$-dimensional Lebes\-gue measure) of $\SS$.

We call $\Lambda$ a \emph{critical lattice} for $\SS$ if $\Lambda$ 
is $\SS$-admissible
and $d(\Lambda )=\Delta (\SS )$.
In a series of papers 
\cite{Mahler1942,Mahler1943,Mahler1946a, Mahler1946b, Mahler1946c} Mahler studied star bodies in $\Rr^2$,
proved that they have critical lattices, and computed their
determinant in various instances.
Later, Mahler picked up the study of star bodies of arbitrary dimension
\cite{Mahler1946d}.
We recall Theorem 8 from this paper, which is Mahler's central 
result on star bodies.

\begin{thm}\label{thm1.1}
Let $\SS$ be a star body in $\Rr^n$ of finite type.
Then $\SS$ has at least one critical lattice.
\end{thm}

The main tool is a compactness result for lattices, also due to Mahler.
We say that a sequence of lattices $\{\Lambda_m\}_{m=1}^{\infty}$ in $\Rr^n$ converges
if we can choose a basis $\av_{m,1}\kdots\av_{m,n}$ of $\Lambda_m$ 
for $m=1,2,\ldots$
such that $\av_j:=\lim_{m\to\infty}\av_{m,j}$ exists
for $j=1\kdots n$ and $\av_1\kdots\av_n$ are linearly
independent. We call the lattice $\Lambda$ with basis
$\av_1\kdots\av_n$ the limit of the sequence  $\{\Lambda_m\}_{m=1}^{\infty}$;
it can be shown that this limit, if it exists, is unique.
Denote by $\|\xv\|$ the Euclidean norm of $\xv\in\Rr^n$.
The following result, which became known as \emph{Mahler's compactness theorem}
or \emph{Mahler's selection theorem} and turned out to be
a valuable tool at various places other than the geometry of numbers,
is Theorem 2 from \cite{Mahler1946d}. 

\begin{thm}\label{thm1.2}
Let $\rho>0$, $C>0$. Then any infinite collection of lattices $\Lambda$
in $\Rr^n$ 
such that $\min\{ \|\xv\|:\, \xv\in\Lambda\setminus\{ 0\}\}\geq\rho$
and $d(\Lambda )\leq C$ has an infinite convergent subsequence.
\end{thm}

We recall the quick deduction of Theorem \ref{thm1.1}.

\begin{proof}[Proof of Theorem \ref{thm1.1}]
By the definition of $\Delta (\SS )$, there is an infinite
sequence $\{ \Lambda_m\}_{m=1}^{\infty}$ of $\SS$-admissible lattices 
such that $\Delta (\SS )\leq d(\Lambda_m)\leq  \Delta (\SS )+1/m$ 
for $m=1,2,\ldots$.
Since $\nullv$ is an interior point of $\SS$,there is $\rho >0$
such that $\{ \xv\in\Rr^n:\, \|\xv\|\leq\rho\}\subseteq\SS$. 
hence $\|\xv\|\geq\rho$ for every non-zero $\xv\in\Lambda_m$
and every $m\geq 1$. Further, the sequence $\{d(\Lambda_m)\}$  
is clearly bounded. So by Theorem \ref{thm1.2}, $\{\Lambda_m\}$ has
a convergent subsequence. After reindexing, we may write this
sequence as $\{\Lambda_m\}_{m=1}^{\infty}$ and denote its limit by
$\Lambda$. We show that $\Lambda$ is a critical lattice for $\SS$.

Choose bases $\av_{m,1}\kdots\av_{m,n}$ of $\Lambda_m$
for $m=1,2,\ldots$ and $\av_1\kdots\av_n$ of $\Lambda$ such that 
$\av_{m,j}\to\av_j$ for $j=1\kdots n$. 
Clearly $d(\Lambda )=\lim_{m\to\infty} d(\Lambda_m)=\Delta (\SS )$.
To prove that $\Lambda$ is $\SS$-admissible, 
take a non-zero $\xv_0\in\Lambda$ and assume it is in the interior of $\SS$.
Then there is $\epsilon >0$ such that all $\xv\in\Rr^n$ with
$\|\xv -\xv_0\|<\epsilon$ are in the interior of $\SS$.
Write $\xv_0=\sum_{i=1}^n z_i\av_i$ with $z_i\in\Zz$,
and then $\xv_m=\sum_{i=1}^n z_i\av_{m,i}$ for $m\geq 1$, so that
$\xv_m\in\Lambda_m\setminus\{ \nullv\}$. For $m$ sufficiently large,
$\|\xv_m-\xv_0\|<\epsilon$, hence $\xv_m$ is in the interior
of $\SS$, which is however impossible since $\Lambda_m$
is $\SS$-admissible.
This completes the proof.
\end{proof}

In \cite{Mahler1946d}, Mahler made a further study of the critical lattices
of $n$-dimensional star bodies.
Among other things he proved \cite[Theorem 11]{Mahler1946d}
that if $\SS$ is any bounded $n$-dimensional star body
and $\Lambda$ a critical lattice for $\SS$, then there are $n$ linearly
independent points of $\Lambda$ lying on the boundary of $\SS$.
If $P_1\kdots P_n$ are such points, then the $2n$ points
$\pm P_1\kdots \pm P_n$ lie on the boundary of $\SS$.
A simple consequence of this is, that any lattice of determinant equal to
$\Delta (\SS )$ has a non-zero point either in the interior or on the 
boundary of $\SS$.
Mahler showed further \cite[Corollary on p. 165]{Mahler1946d} that for any integer $m\geq n$ there exist
an $n$-dimensional star body $\SS$ and a critical lattice $\Lambda$ of $\SS$
having precisely $2m$ points on the boundary of $\SS$.

In an other series of papers on $n$-dimensional star bodies \cite{Mahler1946e}
Mahler introduced the notions of reducible and irreducible star bodies.
A star body $\SS$ is called \emph{reducible} if there is a star body $\SS'$ which is
strictly contained in $\SS$ and for which $\Delta (\SS ')=\Delta (\SS )$,
and otherwise irreducible. An unbounded star body $\SS$ of finite type
is called \emph{boundedly reducible} if there is a bounded star body $\SS'$
contained in $\SS$ such that $\Delta (\SS ')=\Delta (\SS )$.
Mahler gave criteria for star bodies being (boundedly) reducible
and deduced some Diophantine approximation results. To give a flavour
we mention one of these results \cite[Theorem P, p. 628]{Mahler1946e}:

\begin{thm}\label{thm1.3}
There is a positive constant $\gamma$ such that if $\beta_1,\beta_2$
are any real numbers and $Q$ is any number $>1$, then there are integers
$v_1,v_2,v_3$, not all $0$, such that
\begin{align*}
&|v_1v_2(\beta_1v_1+\beta_2v_2+v_3)|\leq\medfrac{1}{7},
\\
&|x_1|\leq Q,\, |x_2|\leq Q,\,
|\beta_1v_1+\beta_2v_2+v_3|\leq \gamma Q^{-2}.
\end{align*}
\end{thm} 

\begin{proof}[Idea of proof]
Let $\SS$ be the set of $\xv=(x_1,x_2,x_3)\in\Rr^3$ given
by $|x_1x_2x_3|\leq 1$.
By a result of Davenport
\cite{Davenport1938}, $\SS$ is a finite type star body and has determinant
$\Delta (\SS )=7$. Mahler \cite[Theorem M, p. 527]{Mahler1946e} 
proved that $\SS$ is in fact boundedly reducible,
which implies that there is $r>0$ such that the star body $\SS'$ given by
$|x_1x_2x_3|\leq 1$ and $\max_{1\leq i\leq 3} |x_i|\leq r$ also has determinant
$7$. Now let $\Lambda$ be the lattice consisting of the points
$\big(rQ^{-1}v_1,rQ^{-1}v_2,7r^{-2}Q^2(\beta_1v_1+\beta_2v_2+v_3)\big)$ with 
$v_1,v_2,v_3\in\Zz$. This lattice has determinant $7$ and so has a non-zero
point in $\SS'$. It follows that Theorem \ref{thm1.3} holds with
$\gamma =r^3/7$. 
\end{proof}

For further theory on star bodies, we refer to Mahler's papers quoted above
and the books of Cassels \cite{Cassels1971}
and Gruber and Lekkerkerker \cite{GruberLekkerkerker1987}. 

\section{Reciprocal convex bodies}\label{section2}

Studies of transference principles such as Khintchine's for systems of 
Diophantine inequalities (see \cite{Mahler1937, Mahler1939a}) led Mahler to consider \emph{reciprocal lattices}
and \emph{reciprocal convex bodies}
(also called polar lattices and polar convex bodies).
We recall some of his results. 
Here and below, for any real vectors $\xv$, $\yv$ of the same
dimension,  
we denote by $\xv\cdot\yv$ their standard inner product,
i.e., for
$\xv =(x_1\kdots x_m),\, \yv =(y_1\kdots y_m)\in\Rr^m$
we put $\xv\cdot\yv :=\sum_{i=1}^m x_iy_i$. Then the Euclidean norm
of $\xv\in\Rr^m$ is $\|\xv\|:=\sqrt{\xv\cdot\xv}$.

Now let $n$ be a fixed integer $\geq 2$.
Given a lattice $\Lambda$ in $\Rr^n$, we define the \emph{reciprocal lattice}
of $\Lambda$ by 
\[
\Lambda^*:=\{ \xv\in\Rr^n:\, \xv\cdot\yv\in\Zz\ \mbox{for all } \yv\in\Lambda\}.
\]
Then $\Lambda^*$ is again a lattice of $\Rr^n$, and $d(\Lambda^*)=d(\Lambda )^{-1}$.
Let $\CC$ be a symmetric convex body in $\Rr^n$, i.e.,
$\CC$ is convex, symmetric about $\nullv$ and compact. The set $\CC$ may be
described alternatively as $\{ \xv\in\Rr^n :\, F(\xv )\leq 1\}$,
where $F$ is a distance function as above, satisfying also the triangle
inequality. We define the reciprocal of $\CC$ by
\[
\CC^*=\{ \xv\in\Rr^n:\, \xv\cdot\yv\leq 1\ \mbox{for all } \yv\in\CC\}.
\]
Then $\CC^*$ is again a symmetric convex body.
Mahler \cite[p. 97, formula (6)]{Mahler1939b} proved the following result for the volumes
of $\CC$ and $\CC^*$.

\begin{thm}\label{thm2.1}
There are $c_1(n),\, c_2(n)>0$ depending only on $n$ with the
following property. If $\CC$ is any symmetric convex body in $\Rr^n$
and $\CC^*$ its reciprocal, then $c_1(n)\leq V(\CC )\cdot V(\CC^*)\leq c_2(n)$.
\end{thm}

Mahler proved this with $c_1(n)=4^n/(n!)^2$ and $c_2(n)=4^n$.
Santal\'{o} \cite{Santalo1949} improved the upper bound to
$c_2(n)=\kappa_n^2$
where $\kappa_n$ is the volume
of the $n$-dimensional Euclidean unit ball 
$B_n:=\{ \xv\in\Rr^n:\, \|\xv\|\leq 1\}$;
this upper bound is attained for $\CC =\CC^*=B_n$.
Bourgain and Milman
\cite{BourgainMilman1987}. 
improved the lower bound to $c_1(n)=c^n\kappa_n^2$ with some absolute constant $c$.
This is probably not optimal. Mahler conjectured that the optimal
value for $c_1(n)$ is $4^n/n!$, which is attained for $\CC$ the unit cube
$\max_i |x_i|\leq 1$ and $\CC^*$ the octahedron $\sum_{i=1}^n |x_i|\leq 1$.  

Recall that the $i$-th successive minimum
$\lambda_i (\CC ,\Lambda )$ of a symmetric convex body $\CC$ in $\Rr^n$
with respect to a lattice $\Lambda$ in $\Rr^n$ is the smallest $\lambda$
such that $\lambda\CC\cap\Lambda$ contains $i$ linearly independent points.
Thus, $\CC$ has $n$ successive minima, 
and by
Minkowski's theorem on successive minima 
\cite{Minkowski1910} one has
\begin{equation}\label{2.minkowski}
\frac{2^n}{n!}\cdot\frac{d(\Lambda )}{V(\CC )}
\leq \lambda_1(\CC ,\Lambda )\cdots \lambda_n(\CC ,\Lambda )\leq
2^n\cdot\frac{d(\Lambda )}{V(\CC )}.
\end{equation}
Mahler \cite[p.100, (A), (B)]{Mahler1939b} proved the following
\emph{transference principle} for reciprocal convex bodies:
 
\begin{thm}\label{thm2.2}
There is $c_3(n)>0$ depending only on $n$ with the following property.
Let $\Lambda$, $\CC$ be a lattice and symmetric convex body in $\Rr^n$,
and $\Lambda^*$, $\CC^*$ their respective reciprocals. Then
\[
1\leq \lambda_i(\CC ,\Lambda )\cdot \lambda_{n+1-i}(\CC^* ,\Lambda^* )\leq c_3(n).
\]
\end{thm}

The lower bounds for the products
$\lambda_i(\CC ,\Lambda )\lambda_{n+1-i}(\CC^* ,\Lambda^* )$ are easy to prove,
and then the upper bounds are obtained by combining 
the lower bound in Theorem \ref{thm2.1} with the upper bound in \eqref{2.minkowski}
and the similar one for $\CC^*$ and $\Lambda^*$.
With his bound for $c_1(n)$, Mahler deduced Theorem \ref{thm2.2}
with $c_3(n)=(n!)^2$.
Using instead the bound for $c_1(n)$ by
Bourgain and Milman, one obtains Theorem \ref{thm2.2} with
$c_3(n)=(c'n)^n$ for some absolute constant $c'$.
Kannan and Lov\'{a}sz \cite{KannanLovasz1988} obtained $\lambda_1(\CC ,\Lambda )\lambda_n^*(\CC^*,\Lambda^*)\leq c''n^2$ with some absolute constant
$c''$. 

Mahler's results led to various applications, among others to inhomogeneous
results. A simple consequence, implicit in
Mahler's paper \cite{Mahler1939b} is the following:

\begin{cor}\label{cor2.3}
There is $c_4(n)>0$ with the following property.
Let $\CC$, $\Lambda$, $\CC^*$ and $\Lambda^*$
be as in Theorem \ref{thm2.2} and suppose
that $\CC^*$ does not contain a non-zero point from $\Lambda^*$.
Then for every $\av\in\Rr^n$ there is $\zv\in\Lambda$ such that
$\av +\zv\in c_4(n)\CC$.
\end{cor}

\begin{proof}[Idea of proof]
Using that the distance function associated with $\CC$ satisfies the
triangle inequality, one easily shows that
for every $\av\in\Rr^n$ there is $\zv\in\Lambda$ with 
$\av +\zv\in n\lambda_n(\CC ,\Lambda )\cdot\CC$. By assumption we have $\lambda_1(\CC^*,\Lambda^*)>1$,
and thus, $\lambda_n(\CC^*,\Lambda^*)< c_3(n)$.
\end{proof}

The second application we mention is a transference principle for systems of Diophantine inequalities.
We define the maximum norm and sum-norm of
$\xv =(x_1\kdots x_n)\in\Rr^n$ by $\|\xv\|_{\infty}:=\max_i |x_i|$
and $\|\xv\|_1:=\sum_{i=1}^n |x_i|$, respectively.
We denote by $A^T$ the transpose of a matrix $A$.

\begin{cor}\label{cor2.4}
Let $m,n$ be integers with $0<m<n$ and
let $A$ be a $(n-m)\times m$-matrix with real entries 
where $m,n$ are integers with $0<m<n$.
Let $\omega$ be the supremum of the reals
$\eta >0$ 
such that there are infinitely many non-zero $\xv\in\Zz^m$ for which there exists
$\yv\in\Zz^{n-m}$ with
\begin{equation}\label{2.1}
\|A\xv- \yv\|_{\infty}\leq \|\xv\|_{\infty}^{-\frac{m}{n-m}(1+\eta )}.
\end{equation}
Further, let $\omega^*$ be the supremum of the reals $\eta^*>0$ for which
there are
infinitely many non-zero $\uv\in\Zz^{n-m}$ for which there exists $\vv\in\Zz^m$
such that
\begin{equation}\label{2.1a}
\|A^T\uv -\vv\|_{\infty}\leq \|\uv\|_{\infty}^{-\frac{n-m}{m}(1+\eta^*)}. 
\end{equation}
Then
\begin{equation}\label{2.1b}
\omega^*\geq \frac{\omega}{(m-1)\omega +n-1},\ \ 
\omega\geq \frac{\omega^*}{(n-m-1)\omega^* +n-1}.
\end{equation}
\end{cor}

These inequalities were proved by Dyson \cite{Dyson1947}.
The special case $m=1$ was established earlier by Khintchine
\cite{Khintchine1925,Khintchine1926}
and became known as \emph{Khintchine's transference principle}.
Jarn\'{i}k \cite{Jarnik1959} proved that
both inequalities are best possible.

\begin{proof}
We prove only the first inequality; then the second follows by symmetry.

Let $Q\geq 1$, $0<\eta <\omega$. Put $\eta^*:=\medfrac{\eta}{(m-1)\eta +n-1}$.
Consider the convex body $\CC_Q$ consisting of the points
$(\xv ,\yv )\in\Rr^m\oplus \Rr^{n-m} =\Rr^{n}$ with $\|\xv\|_{\infty}\leq Q$
and $\|A\xv-\yv\|_{\infty}\leq Q^{-\frac{m}{n-m}(1+\eta )}$.
Denote the successive minima of $\CC_Q$, $\CC_Q^*$, respectively
with respect to $\Zz^n$ by
$\lambda_i(Q)$, $\lambda_i^*(Q)$,
for $i=1\kdots n$.
By the choice of $\eta$, there is a sequence of $Q\to\infty$ such that
$\lambda_1(Q)\leq 1$. Let $Q$ be from this sequence.
The body $\CC_Q$ has volume $V(\CC_Q)\ll Q^{-m\eta}$. The reciprocal body $\CC_Q^*$
of $\CC_Q$ is the
set of $(\uv ,\vv )\in\Rr^{n-m}\oplus\Rr^m$ with $Q\|A^T\uv -\vv\|_1+Q^{-\frac{m}{n-m}(1+\eta )}\|\uv\|_1\leq 1$. Combining Theorem \ref{thm2.2}
with the lower bound in \eqref{2.minkowski}, we infer that 
\[
\lambda_1^*(Q)\ll \lambda_n (Q)^{-1}
\ll\big(V(\CC_Q)\cdot \lambda_1(Q)\big)^{1/(n-1)}
\ll Q^{-m\eta/(n-1)},
\]
where the implied constants depend on $m$ and $n$.
The body $\lambda_1^*(Q)\CC_Q^*$ contains
a non-zero point $(\uv ,\vv )\in\Zz^{n-m}\oplus\Zz^m$, and thus,
\begin{eqnarray*}
&&\|\uv\|_{\infty}\ll Q^{\frac{m}{n-m}(1+\eta )-\frac{m\eta}{n-1}}=:Q',
\\
&&\|A^T\uv-\vv\|_{\infty}\ll Q^{-1-\frac{m\eta}{n-1}}=
Q'^{-\frac{n-m}{m}(1+\eta^*)}.
\end{eqnarray*}
If there is a non-zero $\uv_0\in\Zz^{n-m}$ with $A^T\uv_0 =\vv_0$ for some
$\vv\in\Zz^m$ then \eqref{2.1a} holds with all integer multiples 
of $(\uv_0 ,\vv_0 )$. 
Otherwise, if we let $Q\to\infty$ then $\uv$ runs through
an infinite set.
The first inequality of \eqref{2.1b} easily follows.
\end{proof}

\section{Compound convex bodies}\label{section3}

Mahler extended his theory of reciprocal convex bodies to so-called
\emph{compound convex bodies}, which are in some sense exterior
powers of convex bodies. 

Let again $n\geq 2$ be an integer and $p$ an integer with $1\leq p\leq n-1$.
Put $N:=\medbinom{n}{p}$ and denote by $\II_{n,p}$
the collection of $N$ integer tuples
$(i_1\kdots i_p)$ with $1\leq i_1<\cdots <i_p\leq n$. 
Let $\{\ev_1\kdots\ev_n\}$ be the standard basis
of $\Rr^n$ (i.e., $\ev_i$ has a $1$ on the $i$-th place and zeros elsewhere)
and $\{\widehat{\ev}_1\kdots\widehat{\ev}_N\}$ the standard basis of $\Rr^N$.
We define exterior products of $p$ vectors by means of the multilinear map
$(\xv_1\kdots\xv_p)\mapsto \xv_1\wedge\cdots\wedge\xv_p$ from $(\Rr^n)^p$ to $\Rr^N$, which is such that
$\ev_{i_1}\wedge\cdots\wedge\ev_{i_p}=\widehat{\ev}_j$ for $j=1\kdots N$ if
$(i_1\kdots i_p)$ is the $j$-th tuple of $\II_{n,p}$ in the lexicographic
ordering, and such that $\xv_1\wedge\cdots\wedge\xv_p$
changes sign if two of the vectors are interchanged.

Let $\CC$ be a symmetric body in $\Rr^n$ and $\Lambda$ a
lattice in $\Rr^n$. Then the
\emph{$p$-th compound} $\CC_p$ of $\CC$ is defined as the convex hull
of the points $\xv_1\wedge\cdots\wedge\xv_p\in\Rr^N$ with $\xv_1\kdots\xv_p\in\CC$,
while the $p$-th compound $\Lambda_p$ of $\Lambda$ is the lattice in $\Rr^N$
generated by the points $\xv_1\wedge\cdots\wedge\xv_p$ with
$\xv_1\kdots\xv_p\in\Lambda$. Then $d(\Lambda_p)=d(\Lambda )^P$ where
$P:=\binom{n-1}{p-1}$. Mahler \cite[Theorem 1]{Mahler1955} proved the following
analogue for the volume of the $p$-th compound
of a symmetric convex body. 

\begin{thm}\label{thm3.1}
Let $\CC$ be any symmetric convex body
in $\Rr^n$ and $p$ any integer with $1\leq p\leq n-1$. Then
\[
c_1(n,p)\leq V(\CC_p)\cdot V(\CC )^{-P}\leq c_2(n,p),
\]
where $c_1(n,p)$, $c_2(n,p)$ are positive numbers depending only
on $n$ and $p$.
\end{thm}

\begin{proof}[Idea of proof]
The quotient $V(\CC_p)\cdot V(\CC )^{-P}$ is invariant under linear
transformations, so Theorem \ref{thm3.1} holds 
for ellipsoids, these are the images of the Euclidean unit
ball $B_n:=\{ \xv\in\Rr^n:\, \|\xv\|\leq 1\}$ 
under linear transformations. 
Now the theorem follows for arbitrary symmetric convex bodies $\CC$,
with different $c_1(n,p)$, $c_2(n,p)$,  
by invoking John's theorem \cite{John1948}, 
which asserts that for every symmetric convex body $\CC$
in $\Rr^n$ there is an ellipsoid $\EE$ such that 
$n^{-1/2}\EE\subseteq\CC\subseteq \EE$.
\end{proof}

Mahler \cite[Theorem 3]{Mahler1955} deduced from this the following result on the successive minima of a compound convex body.

\begin{thm}\label{thm3.2}
Let $\CC$ be a symmetric convex body and $\Lambda$ a lattice in $\Rr^n$,
and let $p$ be any integer with $1\leq p\leq n-1$. 
Further, let $\mu_1\kdots \mu_N$,
where $N=\medbinom{n}{p}$, be the products 
$\lambda_{i_1}(\CC,\Lambda )\cdots\lambda_{i_p}(\CC ,\Lambda )$
($(i_1\kdots i_p)\in\II_{n,p}$) in non-decreasing order. Then for the 
successive minima of $\CC_p$ with respect to $\Lambda_p$ we have
\[
c_3(n,p)\leq\frac{\lambda_i(\CC_p,\Lambda_p)}{\mu_i}\leq c_4(n,p)\ \ 
\mbox{for }i=1\kdots N,
\]
where $c_3(n,p),\, c_4(n,p)$ depend on $n$ and $p$ only.
\end{thm}

\begin{proof}[Idea of proof]
Constants implied by $\ll$ and $\gg$ will depend on $n$ and $p$ only. 
Let $\vv_1\kdots\vv_n$ be linearly independent vectors of $\Lambda$
with $\vv_i\in\lambda_i\CC$, where $\lambda_i=\lambda_i(\CC ,\Lambda )$ for
$i=1\kdots n$. Then for each tuple
$(i_1\kdots i_p)\in\II_{n,p}$
we have $\vv_{i_1}\wedge\cdots\wedge\vv_{i_p}\in\lambda_{i_1}\cdots\lambda_{i_p}\CC_p$.
Since the vectors $\vv_{i_1}\wedge\cdots\wedge\vv_{i_p}$ are linearly
independent elements of $\Lambda_p$, it follows that 
$\lambda_i(\CC_p,\Lambda_p)\leq \mu_i$ for $i=1\kdots N$.
On the other hand, by the lower bound of \eqref{2.minkowski} applied to $\CC_p$, $\Lambda_p$
we have $\prod_{i=1}^N \lambda_i(\CC_p,\Lambda_p)\gg d(\Lambda_p)/V(\CC_p)$ and
by the upper bound of \eqref{2.minkowski},  
$\mu_1\cdots\mu_N=(\lambda_1\cdots\lambda_n)^P\ll (d(\Lambda )/V(\CC ))^P$.
By combining this with Theorem \ref{thm3.1}, one
easily deduces Theorem \ref{thm3.2}.
\end{proof}

Mahler's results on compound convex bodies are in fact
generalizations of his results on reciprocal bodies.
To make this precise, let $\CC$ be a symmetric convex body
and $\Lambda$ a lattice in $\Rr^n$ and 
let $\CC^*$, $\Lambda^*$ be their reciprocals.
Then
$\Lambda^*=d(\Lambda )^{-1}\varphi(\Lambda_{n-1})$
where $\varphi$ is the linear map given by
$(x_1\kdots x_n)\mapsto (x_n,-x_{n-1}\kdots (-1)^{n-1}x_1)$.
Further, by an observation of Mahler
\cite[Theorem 4]{Mahler1955},
\[ 
c_5(n)V(\CC)^{-1}\varphi(\CC_{n-1})\subseteq\CC^*\subseteq c_6(n)V(\CC)^{-1}\varphi(\CC_{n-1})
\]
for certain numbers $c_5(n)$, $c_6(n)$ depending only on $n$.
Together with these facts,
Theorems \ref{thm3.1} and \ref{thm3.2} immediately imply
Theorems \ref{thm2.1} and \ref{thm2.2} in a slightly weaker form.

As Mahler already observed in \cite{Mahler1955}, it may be quite
difficult to compute the compounds of a given convex body,
but often one can give an approximation which for applications
is just as good. For instance, let $\av_1\kdots\av_n$ be linearly
independent vectors in $\Rr^n$ and $A_1\kdots A_n$ positive reals,
and consider the parallelepiped
\[
\Pi :=\{ \xv\in\Rr^n:\, |\av_i\cdot \xv |\leq A_i\ \mbox{for } i=1\kdots n\},
\]
where $\cdot$ denotes the standard inner product.
Let $1\leq p\leq n-1$, $N=\medbinom{n}{p}$ and define for $i=1\kdots N$,
\begin{equation}\label{3.0}
\widehat{\av_i}:=\av_{i_1}\wedge\cdots\wedge\av_{i_p},\ \ 
\widehat{A}_i:=A_{i_1}\cdots A_{i_p},
\end{equation}
where $(i_1\kdots i_p)$ is the $i$-th tuple of $\II_{n,p}$
in the lexicographic ordering. Then the \emph{$p$-th pseudocompound}
of $\Pi$ is given by
\[
\widehat{\Pi}_p:=\{ \widehat{\xv}\in\Rr^N:\, |\widehat{\av}_i\cdot \widehat{\xv}|\leq\widehat{A}_i\ \mbox{for } i=1\kdots N\}.
\]
One easily shows (see \cite[p. 377]{Mahler1955}), 
that there are positive numbers $c_7(n,p)$,
$c_8(n,p)$ such that $c_7(n,p)\Pi_p\subseteq \widehat{\Pi}_p\subseteq c_8(n,p)\Pi_p$,
where $\Pi_p$ is the $p$-th compound of $\Pi$.
This implies that Theorem \ref{thm3.2} holds with $\widehat{\Pi}_p$ instead
of $\Pi_p$, with other constants $c_3(n,p),c_4(n,p)$.  

Mahler's results on compound convex bodies turned out to be a very
important tool in Diophantine approximation. First,
it is a crucial ingredient in W.M. Schmidt's
proof of his celebrated Subspace Theorem
\cite{Schmidt1972,Schmidt1980}, and second it has been used
to deduce several transference principles for systems of Diophantine
inequalities. 

We first give a very brief overview of Schmidt's proof
of his Subspace Theorem, focusing on the role of Theorem \ref{thm3.2}.
For the complete proof, 
see \cite{Schmidt1980}.

\begin{subspacetheorem*}
Let $n\geq 2$ and let $L_i(\Xv)=\alpha_{i1}X_1+\cdots +\alpha_{in}X_n$
$(i=1\kdots n)$ be linearly independent linear forms with algebraic coefficients
in $\Cc$. Further, let $\delta >0$. Then the set of solutions of
\begin{equation}\label{3.1}
|L_1(\xv )\cdots L_n(\xv )|\leq \|\xv\|^{-\delta}\ \ \mbox{in } \xv\in\Zz^n
\end{equation}
is contained in finitely many proper linear subspaces of $\Qq^n$.
\end{subspacetheorem*}

\begin{proof}[Outline of the proof]
We can make a reduction to the case that $L_1\kdots L_n$ all have real
algebraic coefficients by replacing each $L_i$ by its real or imaginary part,
such that the resulting linear forms are linearly independent. 
Further, after a normalization we arrange that these linear forms
have determinant $1$.
So henceforth
we assume that the coefficients of $L_1\kdots L_n$ are real algebraic,
with $\det (L_1\kdots L_n)=1$.
Next, it suffices to consider only $\xv\in\Zz^n$ with 
$L_i(\xv )\not= 0$ for $i=1\kdots n$. 

Now let $\xv\in\Zz^n$ be a solution
of \eqref{3.1} and put 
\begin{eqnarray*}
&&A_i:=|L_i(\xv )|/|L_1(\xv )\cdots L_n(\xv )|^{1/n}\ (i=1\kdots n),\\
&&\Av :=(A_1\kdots A_n),\ \
Q(\Av ):=\max (A_1\kdots A_n).
\end{eqnarray*}
With this choice, $A_1\cdots A_n=1$.
Assuming that $\|\xv\|$ is sufficiently large, there is a fixed $D>0$ independent
of $\xv$, 
such that $\|\xv\|^{-D}\leq |L_i(\xv )|\leq \|\xv\|^D$
for $i=1\kdots n$.
Hence
$Q(\Av )\leq \|\xv\|^{2D}$. 
Write $L_i(\Xv )=\av_i\cdot\Xv$ where $\av_i$ is the vector 
of coefficients of $L_i$ and consider 
the parallelepiped
\begin{equation}\label{3.2}
\Pi (\Av ):=\{ \yv\in\Rr^n:\ |\av_i\cdot\yv|\leq A_i\ \mbox{for } i=1\kdots n\}.
\end{equation}
Since $|L_1(\xv )\cdots L_n(\xv )|^{1/n}\leq \|\xv\|^{-\delta/ n}\leq Q(\Av )^{-\delta_1}$ with $\delta_1:=\delta /2nD$, we have
\[
\xv\in Q(\Av )^{-\delta_1}\Pi(\Av ).
\]
Let $T(\Av )$ denote
the vector space generated by
$Q(\Av )^{-\delta_1}\Pi(\Av )\cap\Zz^n$. So $\xv\in T(\Av )$. 
It clearly suffices to show the following:
\\[0.15cm]
\emph{for every $\delta_1>0$
there is a finite collection $\{ T_1\kdots T_t\}$ of proper linear subspaces
of $\Qq^n$ such that for every $n$-tuple $\Av$ of positive reals 
with $A_1\cdots A_n=1$, the vector space
$T(\Av )$ is contained in one of $T_1\kdots T_t$.}
\vskip0.15cm

Assume that this assertion is false. 
Pick many tuples $\Av_1\kdots\Av_m$ such that the spaces
$T^{(i)}:=T(\Av_i)$ ($i=1\kdots m$) are all different. Then one can construct
a polynomial in $m$ blocks of $n$ variables $\Xv_1\kdots\Xv_m$  with integer coefficients, which is homogeneous in each block and divisible 
by high powers of $L_i(\Xv_j)$, for $i=1\kdots n$, $j=1\kdots m$. 
All partial derivatives of this polynomial of order up to a certain bound
have absolute value $<1$, hence are $0$,
at many integral points
of $T^{(1)}\times\cdots\times T^{(m)}$. Then by extrapolation, it follows
that this polynomial vanishes with high multiplicity on all of
$T^{(1)}\times\cdots\times T^{(m)}$. Now one would like to apply a 
non-vanishing result implying that this is impossible, but such a 
result can been proved only if the dimensions of
$T^{(1)}\kdots T^{(m)}$ are equal to $n-1$. So the above argument
works only for those tuples $\Av$ for which $\dim T(\Av )=n-1$,
that is, for which the $(n-1)$-th successive minimum of $\Pi (\Av)$
with respect to $\Zz^n$
is at most $Q(\Av )^{-\delta_1}$.

Now Schmidt could make his proof of the Subspace Theorem work 
for arbitrary tuples $\Av$ by means of an ingenious argument,
in which he constructs from the parallelepiped $\Pi (\Av )$ a
new parallelepiped $\widehat{\Pi}(\widehat{\Bv})$, in general of larger dimension $N$, with $\widehat{\Bv}=(\widehat{B}_1\kdots\widehat{B}_N)$
satisfying $\widehat{B}_1\cdots\widehat{B}_N=1$,
of which the $(N-1)$-th successive minimum with respect to $\Zz^N$
is small.
In this construction, Mahler's results on compound convex bodies play
a crucial role.

In what follows, constants implies by $\ll$, $\gg$, $\asymp$ will depend
only on $n$, $\delta_1$ and $L_1\kdots L_n$, while $\delta_2,\delta_3,\ldots$
will denote positive numbers depending only on $\delta_1$ and $n$.
Denote the successive minima of $\Pi (\Av )$ with respect to $\Zz^n$ by
$\lambda_1\kdots\lambda_n$. Then clearly,
\[ 
\lambda_1\leq  Q(\Av )^{-\delta_1}.
\]
Further, by \eqref{2.minkowski},
\begin{equation}\label{3.4}
\lambda_1\cdots\lambda_n\asymp 1. 
\end{equation}
Let $k$ be the largest
index with $\lambda_k\leq Q(\Av )^{-\delta_1}$. Then \eqref{3.4}
implies that $\lambda_n\gg Q(\Av )^{k\delta_1/(n-k)}$. Hence 
there is $p$ with $k\leq n-p\leq n-1$ such that
$\lambda_{n-p}/\lambda_{n-p+1}\ll Q(\Av )^{-\delta_2}$.
Let $S(\Av )$ be the vector space generated by
$\lambda_{n-p}\Pi (\Av )\cap\Zz^n$. This space contains $T(\Av )$.
So it suffices to prove that \emph{$S(\Av )$ runs through a finite collection
of proper linear subspaces of $\Qq^n$.}

Let $N:=\medbinom{n}{p}$ and consider the $p$-th pseudocompound
\[
\widehat{\Pi}_p(\widehat{\Av})=\{\widehat{\yv}\in\Rr^N:\, |\widehat{\av}_i\cdot \widehat{\yv}|\leq \widehat{A}_i\ \ \mbox{for } i=1\kdots N\}.
\]
Denote by $\widehat{\lambda}_1\kdots\widehat{\lambda}_N$ the successive minima
of $\widehat{\Pi}_p(\widehat{\Av})$ with respect to $\Zz^n$.
Then by Theorem \ref{thm3.2} we have for the last two minima,
$\widehat{\lambda}_{N-1}\asymp \lambda_{n-p}\lambda_{n-p+2}\cdots\lambda_n$,
$\widehat{\lambda}_N\asymp\lambda_{n-p+1}\cdots\lambda_n$.
Hence
\begin{equation}\label{3.6}
\widehat{\lambda}_{N-1}/\widehat{\lambda}_N\ll \lambda_{n-p}/\lambda_{n-p+1}\ll Q(\Av )^{-\delta_2}.
\end{equation}
Moreover, by \eqref{3.4}, Theorem \ref{thm3.2} we have
\begin{equation}\label{3.7}
\widehat{\lambda}_1\cdots\widehat{\lambda}_N\asymp 1.
\end{equation}

We still need one reduction step. By a variation on a result of
Davenport, proved by Schmidt (see e.g., \cite[p. 89]{Schmidt1980}),
for every choice of reals $\rho_1\kdots\rho_N$ with
\[
\rho_1\geq\cdots \geq \rho_N>0,\ \ \ 
\rho_1\widehat{\lambda}_1\leq\cdots\leq \rho_N\widehat{\lambda}_N,
\ \ \
\rho_1\cdots\rho_N=1,
\]
there is a permutation $\sigma$ of $1\kdots N$ such that the parallelepiped
\[
\widehat{\Pi}_p(\widehat{\Bv})=
\{ \widehat{\yv}\in\Rr^N:\, |\widehat{\av}_i\cdot\widehat{\yv}|\leq \widehat{B}_i\ 
\mbox{for } i=1\kdots N\},
\]
where $\widehat{B}_i:=\rho_{\sigma (i)}^{-1}\widehat{A}_i$ for $i=1\kdots N$, 
has successive minima $\widehat{\lambda}_i'\asymp\rho_i\widehat{\lambda}_i$
for $i=1\kdots N$. Now with the choice
\[
\rho_i=c/\widehat{\lambda}_i\  (i=1\kdots N-1),\ \ \rho_N=c/\widehat{\lambda}_{N-1}
\]
where 
\[
c=(\widehat{\lambda_1}\cdots \widehat{\lambda_N})^{1/N}(\widehat{\lambda}_{N-1}/\widehat{\lambda}_{N})^{1/N}
\]
has been chosen to make $\rho_1\cdots\rho_N=1$,
we obtain $\widehat{\lambda}_{N-1}'\ll c\ll Q(\Av )^{-\delta_3}$
in view of  \eqref{3.6},\eqref{3.7}.
One can show that
\[ 
Q(\widehat{\Bv}):=\max (\widehat{B}_1,\kdots \widehat{B}_N)\ll Q(\Av )^d
\]
with $d$ depending
only on $n$ and $p$. Thus, $\widehat{\lambda}_{N-1}'\ll Q(\widehat{\Bv})^{-\delta_4}$.
Further,
\[
\widehat{B}_1\cdots\widehat{B}_N=\rho_1\cdots\rho_N(A_1\cdots A_n)^{\binom{n-1}{p-1}}=1.
\] 

Now by means of the argument sketched above, with the construction of the polynomial
and the application of the non-vanishing result, one can show that 
if $\Av=(A_1\kdots A_n)$ runs through the tuples of positive reals
with $A_1\cdots A_n=1$, then
the vector space $T(\widehat{\Bv})$ generated by $\widehat{\lambda}_{N-1}'\widehat{\Pi}_p(\widehat{\Bv})\cap\Zz^N$
runs through a finite collection. One can show that $T(\widehat{\Bv})$
uniquely determines the space $S(\Av )$. Hence $S(\Av )$ runs through
a finite collection.
This proves the Subspace Theorem.
\end{proof}

We should mention here that Faltings and W\"{u}stholz \cite{FaltingsWuestholz1994}
gave a very different proof of the Subspace Theorem,
avoiding geometry of numbers but using instead some involved
algebraic geometry.

Mahler's results on compound convex bodies have been applied at
various other places, in particular to obtain generalizations
of Khintchine's transference principle and Corollary \ref{cor2.4}.
Many of these results can be incorporated into the
\emph{Parametric Geometry of Numbers}, a recent theory which was initiated by
Schmidt and Summerer \cite{SchmidtSummerer2009,
SchmidtSummerer2013}. The general idea is as follows.
Let $\mu_1\kdots\mu_n$ be fixed reals which we normalize so that
$\mu_1+\cdots +\mu_n=0$
and consider the parametrized class of convex bodies in $\Rr^n$,
\[
\CC (q):=\{ \xv =(x_1\kdots x_n)\in\Rr^n:\, |x_i|\leq e^{\mu_iq}\ \mbox{for }
i=1\kdots n\}\ \ (q>0).
\]
Further, let $\Lambda$ be a fixed lattice in $\Rr^n$ and 
$\lambda_1(q)\kdots\lambda_n(q)$ the successive minima of $\CC (q)$
with respect to $\Lambda$. Then one would like to study these
successive minima as functions of $q$. In particular, one is interested
in the quantities
\begin{equation}\label{3.3}
\left\{
\begin{array}{l}
\displaystyle{\underline{\varphi}_i=\underline{\varphi}_i(\Lambda ,\bm{\mu}):=\liminf_{q\to\infty}(\log\lambda_i(q))/q},
\\[0.2cm]
\displaystyle{\overline{\varphi}_i=\overline{\varphi}_i(\Lambda ,\bm{\mu}):=\limsup_{q\to\infty}(\log\lambda_i(q))/q}
\end{array}\right.\ \ \
(i=1\kdots n).
\end{equation}
That is, $\underline{\varphi}_i$ is the infimum of all $\eta$ such that there
are arbitrarily large $q$ for which the system of inequalities
\begin{equation}\label{3.a}
|x_1|\leq e^{(\mu_1+\eta )q}\kdots |x_n|\leq e^{(\mu_n+\eta )q}
\end{equation}
is satisfied by $i$ linearly independent points from $\Lambda$,
while $\overline{\varphi}_i$ is the infimum of all $\eta$ such that for 
every sufficiently large $q$, system \eqref{3.a} is satisfied by
$i$ linearly independent points from $\Lambda$.
The quantities $\underline{\varphi}_i$, $\overline{\varphi}_i$ are finite, since if $\mu >\max_j |\mu_j|$,
then for every sufficiently large $q$, the body $e^{\mu q}\CC (q)$ contains $n$ 
linearly independent points from $\Lambda$, while $e^{-\mu q}\CC (q)$ 
does not contain
a non-zero point of $\Lambda$.

In case that $\Lambda$ is an algebraic lattice, i.e., if it is generated
by vectors with algebraic coordinates, then by following the proof of the 
Subspace Theorem one can show that $\underline{\varphi}_i=\overline{\varphi}_i$
for $i=1\kdots n$, i.e., the limits exist (this is a special case of
\cite[Theorem 16.1]{EvertseFerretti2013}, but very likely this was known
before). However, for non-algebraic lattices $\Lambda$ it may happen that $\underline{\varphi}_i<\overline{\varphi}_i$ for some $i$.

Many of the Diophantine approximation
exponents that have been introduced during the last decades can be expressed
in terms of the quantities $\underline{\varphi}_i$, 
$\overline{\varphi}_i$, and thus, results for these exponents can be translated
into results for the $\underline{\varphi}_i$, $\overline{\varphi}_i$.
For instance, let $A$ be a real $(n-m)\times m$-matrix with $1\leq m<n$,
and take
\begin{eqnarray*}
&&\Lambda =\{ (\xv ,A\xv-\yv):\, \xv\in\Zz^m,\, \yv\in\Zz^{n-m}\},
\\
&&\mu_1=\cdots =\mu_m=n-m,\ \ \mu_{m+1}=\cdots =\mu_n=m.
\end{eqnarray*}
Define $\underline{\varphi}_i(A):=\underline{\varphi}_i(\Lambda ,\bm{\mu})$
for this $\Lambda$ and $\bm{\mu}$. Then for the quantities $\omega$, $\omega^*$
from Corollary \ref{cor2.4} we have
\[
\underline{\varphi}_1(A)=-\frac{(n-m)^2\omega}{n+(n-m)\omega},\ \ \ 
\underline{\varphi}_1(A^T)=-\frac{m^2\omega^*}{n+m\omega^*},
\]
and the inequalities \eqref{2.1b} become
\[
\underline{\varphi}_1(A^T)\leq\medfrac{1}{n-1}\cdot\underline{\varphi}_1(A),
\ \ \ 
\underline{\varphi}_1(A)\leq\medfrac{1}{n-1}\cdot\underline{\varphi}_1(A^T).
\]

Studying the successive minima functions
$\lambda_i(q)$ for arbitrary lattices $\Lambda$ and reals $\mu_1\kdots\mu_n$
is probably much too hard.
In their papers \cite{SchmidtSummerer2009, SchmidtSummerer2013}
Schmidt and Summerer considered the special case
\begin{equation}\label{3.b}
\left\{\begin{array}{l}
\Lambda =\{(x,\xi_1x-y_1,\cdots\xi_{n-1}x-y_{n-1}):\, x,y_1\kdots y_{n-1}\in\Zz\},
\\
\mu_1=n-1,\ \mu_2=\cdots =\mu_n=-1,
\end{array}\right.
\end{equation}
where $\xi_1\kdots\xi_{n-1}$ are reals such that $1,\xi_1\kdots\xi_{n-1}$
are linearly independent over $\Qq$. 
That is, they considered the system of
inequalities
\[
|x|\leq e^{(n-1)q},\ \  |\xi_ix-y_i|\leq e^{-q}\ \ (i=1\kdots n-1).
\]
Let $\underline{\varphi}_i,\overline{\varphi}_i$ be the quantities 
defined in \eqref{3.3}, with $\Lambda ,\bm{\mu}$ as in \eqref{3.b}.
In \cite{SchmidtSummerer2009},
Schmidt and Summerer showed among other things that for every $i\in\{ 1\kdots n-1\}$ there
are arbitrarily large $q$ such that $\lambda_{i+1}(q)=\lambda_i(q)$.
As a consequence, $\underline{\varphi}_{i+1}\geq\overline{\varphi}_i$
for $i=1\kdots n-1$. They deduced several other
algebraic inequalities for the numbers $\underline{\varphi}_i,\overline{\varphi}_i$.

In \cite{SchmidtSummerer2013}, Schmidt and Summerer continued their
research and studied in more detail
the functions
\[
L_i(q):=\log\lambda_i(q)\ \ (i=1\kdots n).
\]
To this end, they introduced a class
of $n$-tuples of continuous, piecewise linear functions on $(0,\infty )$
with certain properties, the so-called \emph{$(n,\gamma )$-systems}.
The key argument in their proof is, that there is an $(n,\gamma )$-system
$(P_1(q)\kdots P_n(q))$ 
such that $|L_i(q)-P_i(q)|\leq c(n)$ for $i=1\kdots n$, $q>0$, 
where $c(n)$ depends on $n$ only. In the construction
of these functions, essential use is made of Mahler's results on
compound convex bodies. Indeed, for $p=1\kdots n-1$ let $\CC^{(p)}(q)$ be the $p$-th pseudocompound
of $\CC (q)$ and let $e^{M_p(q)}$ be the first minimum of $\CC^{(p)}(q)$
with respect to the $p$-th compound $\Lambda_p$ of $\Lambda$.
Further, put $M_0(q)=M_n(q):=0$.
Schmidt and Summerer showed that the functions
$P_i(q):=M_i(q)-M_{i-1}(q)$  ($i=1\kdots n$)
form an $(n,\gamma )$-system.
Theorem \ref{thm3.2} implies that there is $c(n)>0$
such that $|L_i(q)-P_i(q)|\leq c(n)$
for $i=1\kdots n$, $q>0$. 
It is important that $P_1(q)+\cdots +P_n(q)=0$, while for the original
functions $L_1(q)\kdots L_n(q)$ one knows only that
their sum is bounded.
It is clear that for $i=1\kdots n$ we have $\underline{\varphi}_i=\underline{\pi}_i$,  
$\overline{\varphi}_i=\overline{\pi}_i$ where
$\underline{\pi}_i:=\liminf_{q\to\infty} P_i(q)/q$ and
$\overline{\pi}_i:=\limsup_{q\to\infty} P_i(q)/q$. 

Schmidt and Summerer analyzed
$(n,\gamma )$-systems, which involved basically
combinatorics and had no connection with geometry of numbers anymore.
As a result of their
(fairly difficult) analysis they obtained several algebraic inequalities
for  $\underline{\pi}_i,\, \overline{\pi}_i$ $(i=1\kdots n)$.
These imply of course the same
inequalities for
$\underline{\varphi}_i$, $\overline{\varphi}_i$ $(i=1\kdots n)$.
This led to new proofs of older results and also various new results.

For instance, it is an easy consequence of Minkowski's theorem on successive
minima that
\[
(n-1)\underline{\varphi}_1+\overline{\varphi}_n\leq 0,\ \ 
(n-1)\overline{\varphi}_n+\underline{\varphi}_1\geq 0.
\]
Schmidt and Summerer \cite[bottom of p. 55]{SchmidtSummerer2013} improved this to
\[
(n-1)\underline{\varphi}_1+\overline{\varphi}_n\leq \overline{\varphi}_1(n-\underline{\varphi}_1+\overline{\varphi}_n),\ \ 
(n-1)\overline{\varphi}_n+\underline{\varphi}_1\geq \underline{\varphi}_n(n-\overline{\varphi}_n+\underline{\varphi}_1).
\]

Recently, Roy \cite{Roy2015} showed that the functions
$L_1(q)\kdots L_n(q)$ considered by Schmidt and Summerer can be approximated very
well by piecewise linear functions from a more restrictive class,
the \emph{$(n,0)$-systems}. This smaller class may be more easy to analyse
than the $(n,\gamma )$-systems and may perhaps lead to new insights
in the functions $L_i(q)$.

\end{document}